\documentclass[a4paper,12pt]{article}
\usepackage{amsmath,amsthm,amsfonts,shuffle}
\usepackage{fullpage,comment,graphicx,epstopdf}
\usepackage{multicol}
\usepackage[all,knot]{xy}

\input{xy}
\xyoption{all}
\xyoption{matrix}

\theoremstyle{plain}

\newtheorem{teo}{Theorem}
  \newtheorem{lem}[teo]{Lemma}
  \theoremstyle{remark}
  \newtheorem{rem}[teo]{Remark}

  \newtheorem{question}[teo]{Question}
  
\theoremstyle{definition}
  
  \theoremstyle{definition}
  
    \newtheorem{defi}[teo]{Definition}

  \newtheorem{ex}[teo]{Example}
  \theoremstyle{plain}
  \newtheorem{prop}[teo]{Proposition}
  \theoremstyle{plain}
  
  \newtheorem{coro}[teo]{Corollary}
  \newtheorem{conj}[teo]{Conjecture}

\def\t{\triangleleft}

\def\ot{\otimes}
\def\wh{\widehat}
\def\wt{\widetilde}
\def\Hom{\mathrm{Hom}}

\def\Ker{\mathrm{Ker}}
\def\Der{\mathrm{Der}}
\def\Sh{\mathrm{Sh}}
\def\gr{\mathrm{gr}}
\def\dd{\hbox{disdeg}}
\def\B{\mathfrak{B}}

\def\id{\mathrm{Id}}

\def\N{\mathbb{N}}
\def\K{\mathbb{K}}
\def\BB{\mathbb{B}}
\def\SS{\mathbb{S}}

\def\Z{\mathbb{Z}}
\def\Q{\mathbb{Q}}
\def\om{\omega}

\def\ra{\overline}

\def\m1{^{ \hbox{\small{-}}1}}

\title{A differential bialgebra associated to a set theoretical solution of the Yang-Baxter equation}

\author{Marco A. Farinati\thanks{Member of CONICET. Partially supported by
PIP 11220110100800CO, and UBACYT 20021030100481BA, mfarinat@dm.uba.ar.} 
\ and Juliana Garc\'ia Galofre\thanks{Partially supported by 
PIP 11220110100800CO and UBACYT 20021030100481BA,
jgarciag@dm.uba.ar }
}

\begin{document}
\maketitle
\begin{abstract}

For a set theoretical solution of the Yang-Baxter equation $(X,\sigma)$, we
define a d.g. bialgebra $B=B(X,\sigma)$, containing the semigroup algebra
$A=k\{X\}/\langle xy=zt : \sigma(x,y)=(z,t)\rangle$,
 such that $k\ot_AB\ot_Ak$ and $\Hom_{A-A}(B,k)$ are respectively
the homology and cohomology complexes computing 
biquandle homology and cohomology defined in 
\cite{CEGN, CJKS} and other  generalizations of cohomology 
of rack-quanlde case (for example defined in \cite{CES}). 
This algebraic structure allow 
us to show the existence of an associative product in the cohomology of biquandles,
and a comparison map with Hochschild (co)homology of the algebra $A$.

\end{abstract}

\section{Introduction}

A quandle is a set $X$ together with a binary  operation
$*:X\times X\to X$ satisfying certain conditions (see definition
on example \ref{exrack} below), it  generalizes the operation
of conjugation on a group, but also  is an algebraic structure that behaves well with respect to
Reidemeister moves, so it is very useful for defining knot/links invariants. Knot theorists
have defined a cohomology theory for quandles (see \cite{CJKS}
and \cite{tCES}) in such a way that 2-cocycles give rise to knot invariants by means of 
the so-called state-sum procedure. 
Biquandles are generalizations of quandles in the sense that quandles give rise to solutions
of the Yang-Baxter equation by setting $\sigma(x,y):=(y,x*y)$. For biquandles there is also a cohomology theory and state-sum procedure for producing knot/links invariants
(see \cite{CES}).

In this work, for a set theoretical solution of the Yang-Baxter equation $(X,\sigma)$, we
define a d.g. algebra $B=B(X,\sigma)$, containing the semigroup algebra
$A=k\{X\}/\langle xy=zt : \sigma(x,y)=(z,t)\rangle$,
 such that $k\ot_AB\ot_Ak$ and $\Hom_{A-A}(B,k)$ are respectively
the standard homology and cohomology complexes attached to general set theoretical
solutions of the Yang-Baxter equation. 
We prove that this d.g. algebra has a natural structure of d.g. {\em bialgebra}
(Theorem \ref{teobialg}). Also, depending on properties of the solution $(X,\sigma)$
(square free, quandle type, biquandle, involutive,...) this d.g. bialgebra $B$ has 
natural (d.g. bialgebra) quotients, giving  rise to the standard
sub-complexes computing quandle cohomology (as sub-complex of rack homology),
 biquandle cohomology, etc.

As a first consequence of our construction, we give a very simple and  purely algebraic proof
of the existence of a cup product in cohomology. This was known for rack cohomology
(see \cite{Cl}), the proof was based on topological methods, but it was unknown for biquandles
or general solutions of the Yang-Baxter equation.
A second consequence is the existence of a comparison map between Yang-Baxter (co)homology
and Hochschild (co)homology of the semigroup algebra $A$. Looking carefully this comparison map
we prove that it factors through a complex of "size" $A\ot \B\ot A$, where
$\B$ is the Nichols algebra associated to the solution $(X,-\sigma)$.
This result leads to new questions, for instance when $(X,\sigma)$ is
 involutive (that is $\sigma^2=\id$) and the characteristic
is zero we show that this complex is acyclic (Proposition \ref{propinvo}), we wander if
 this is true in any other characteristic, and for non necessarily involutive solutions.

{\bf Acknowledgements:}
The first author wishes to thank Dominique Manchon for fruitful discussion  during a visit to
 Laboratoire de math\'ematiques de l'Universit\'e Blaise Pascal where  a preliminary version of 
the bialgebra $B$ for racks came up. He also want to thanks Dennis Sullivan for very pleasant
stay in Stony Brook where the contents of this work was discussed in detail, in particular,
the role of Proposition \ref{fnormal} in the whole construction.

\subsection{Basic definitions}
A set theoretical solution of the Yang-Baxter equation (YBeq) is a pair $(X, \sigma)$ where $\sigma: X\times X\rightarrow X\times X$
is a bijection satisfying 
\[
 (\id\times\sigma)(\sigma\times \id)(\id\times\sigma)=(\sigma\times \id)(\id\times\sigma)(\sigma\times \id):X\times X\times X\rightarrow X\times X\times X
\]

If $X=V$ is a $k$-vector space and $\sigma$ is a linear bijective map satisfying YBeq 
then it is called a braiding on $V$.

\begin{ex}\label{exrack} A set $X$ with a binary operation $\t:X\times X\rightarrow X\times X$ is called a rack if 
\begin{itemize}
 \item $-\t x:X\rightarrow X$ is a bijection $\forall x\in X$ and 
 \item $(x\t y)\t z=(x\t z)\t (y\t z)$ $\forall x,y,z \in X$. 
\end{itemize}

$x\t y $ is usually denoted by $x^y$.

If $X$ also verifies that $x\t x=x$ then $X$ is called a {\em quandle}.

An important example of rack is $X=G$ a group, $x\t y=y^{-1}xy$.

If $(X,\t)$ is a rack, then \[
                             \sigma(x,y)=(y, x\t y)
                            \]
  is a set theoretical solution of the YBeq.
\end{ex}  

  Let $M=M_X$ be the monoid freely generated in $X$  with relations \[
                                                                 xy=zt
                                                                \]
$\forall x,y,z,t$ such that $\sigma(x,y)=(z,t)$. Denote $G_X$ the group with the same generators and relations.
For example, when 
$\sigma= \text{flip}$ then $M=\N_0^{(X)}$ and $G_X=\Z_0^{(X)}$. If $ \sigma=\id$ then $M$ is the free (non abelian)
monoid in $X$. If $\sigma$ comes from a rack $(X,\t)$ then $M$ is the monoid with relation   
$xy=y(x\t y)$ and $G_X$ is the group with relations $x\t y=y^{-1}xy$.

\section{A d.g. bialgebra associated to $(X,\sigma)$}
Let $k$ be a commutative ring with 1. 
Fix $X$ a set, and $\sigma:X\times X\to X\times X$ a solution of the YBeq.
Denote $A_\sigma(X)$, or simply $A$ if $X$ and $\sigma$ are understood,
 the quotient of the free $k$ algebra on generators $X$
modulo the ideal generated by elements of the form $xy-zt$ whenever $\sigma(x,y)=(z,t)$:
\[
A:=k\langle X\rangle/\langle xy-zt : x,y\in X,\ (z,t)=\sigma(x,y)\rangle=k[M]
\]
It can be easily seen that
$A$ is a $k$-bialgebra declaring $x$ to be grouplike for any $x\in X$,
 since $A$ agrees with the semigroup-algebra on $M$ (the monoid
freely generated by $X$ with relations $xy\sim zt$).
If one considers  $G_X$, the group freely generated by
$X$ with relations $xy=zt$, then $k[G_X]$ is the (non commutative) localization of $A$, where one has inverted the
elements of $X$.
An example of $A$-bimodule that will be used later, which is actually a $k[G_X]$-module, is
 $k$ with $A$-action determined on generators by
\[
x\lambda y=\lambda, \ \forall x,y\in X,\ \lambda\in k
\]
We define $B(X,\sigma)$ (also denoted by $B$) the algebra freely generated 
by three copies of $X$, denoted $x$, $e_x$ and $x'$,
with  relations as follows:
whenever $\sigma(x,y)=(z,t)$ we have
\begin{itemize}
\item $ xy\sim zt$ , $xy'\sim z't$, $x'y'\sim z't'$
\item $ xe_{y}\sim e_zt$, $ e_xy'\sim z'e_{t}$
\end{itemize}
Since the relations are homogeneous, $B$ is a graded algebra  declaring
 \[
  |x|=|x'|=0,\ \ |e_x|=1
 \]

\begin{teo}\label{teobialg}
 The algebra $B$ admits the structure of a differential graded bialgebra, with $d$ the unique superderivation satisfying
  \[
  d(x)=d(x')=0,\ \
  d(e_x)=x-x'
 \]
and comultiplication determined by 
\[
\Delta(x)=x\ot x,\
\Delta(x')=x'\ot x',\
\Delta(e_x)=x'\ot e_x+e_x\ot x
\]
\end{teo}
By differential graded bialgebra we mean that the differential is both a derivation with respect to multiplication, and 
coderivation with respect to comultiplication.
\begin{proof}
In order to see that $d$ is well-defined as super derivation, one must check that the relations
 are compatible with $d$. The first relations
are easier since
\[
d(xy-zt)=
d(x)y+xd(y)-d(z)t-zd(t)=0+0-0-0=0
\]
and similar for the others
(this implies that $d$ is $A$-linear and $A'$-linear). For the rest of the relations:
\[
d(xe_{y}-e_zt)=
xd(e_y)-d(e_z)t
=
x(y-y')-(z-z')t
\]
\[
=xy-zt-(xy'-z't)=0
\]
\[
 d(e_xy'-z'e_{t})
=(x-x')y'-z'(t-t')
=xy'-z't -(x'y'-z't')=0
\]
It is clear now that $d^2=0$ since $d^2$ vanishes on generators.
In order to see that $\Delta$ is well defined, we compute

\[
 \Delta(xe_y-e_zt)
 =
(x\ot x)( y'\ot e_y+e_y\ot y)
-(z'\ot e_z+e_z\ot z)(t\ot t)
\]
\[=
xy'\ot xe_y+xe_y\ot xy
-z't\ot e_zt-e_zt\ot zt
\]
and using the relations we get
\[=
xy'\ot xe_y+xe_y\ot xy
-xy'\ot xe_y-xe_y\ot xy=0
\]
similarly
\[
 \Delta(x'e_y-e_zt')
 =
(x'\ot x')( y'\ot e_y+e_y\ot y)
-(z'\ot e_z+e_z\ot z)(t'\ot t')
\]
\[
=
x'y'\ot x'e_y+x'e_y\ot x'y
-z't'\ot e_zt'-e_zt'\ot zt'
\]
\[=
x'y'\ot x'e_y+x'e_y\ot x'y
-x'y'\ot x'e_y-x'e_y\ot x'y=0
\]
This proves that $B$ is a bialgebra, and $d$ is (by construction) a derivation.
Let us see that it is also a coderivation:
\[
(d\ot 1+1\ot d)(\Delta(x))=
(d\ot 1+1\ot d)(x\ot x)=0=\Delta(0)=\Delta(dx)
\]
for  $x'$ is the same. For $e_x$:
\[
(d\ot 1+1\ot d)(\Delta(e_x))=
(d\ot 1+1\ot d)(x'\ot e_x+e_x\ot x)
\]
\[=
x'\ot (x-x')+(x-x')\ot x
=
x'\ot x-x'\ot x'+x\ot x
-x'\ot  x
\]
\[=
-x'\ot x'+x\ot x
=\Delta(x-x')=\Delta(de_x)
\]
 \end{proof}
 \begin{rem}
  $\Delta$ is coassociative.
 \end{rem}

For a particular element of the form
$b=e_{x_1}\dots e_{x_n}$, the formula for $d(b)$ can be computed as follows:
\[  
d(e_{x_1}\dots e_{x_n})=\sum_{i=1}^{n} (-1)^{i+1}e_{x_1}\dots e_{x_{i-1}}d(e_{x_i}) e_{x_{i+1}}\dots e_{x_n}
\]
\[
=\sum_{i=1}^{n} (-1)^{i+1}e_{x_1}\dots e_{x_{i-1}}(x_i-x'_i)e_{x_{i+1}}\dots e_{x_n}
\]
\[=
\overbrace{\sum_{i=1}^{n} (-1)^{i+1}e_{x_1}\dots e_{x_{i-1}}x_ie_{x_{i+1}}\dots e_{x_n}}^{I}
 -\overbrace{\sum_{i=1}^{n} (-1)^{i+1}e_{x_1}\dots e_{x_{i-1}}x'_ie_{x_{i+1}}\dots e_{x_n}}^{II}
\]
If one wants to write it in a normal form (say, every $x$ on the right, every $x'$ on the left,
and the $e_x$'s in the middle), then one should use the relations in $B$: this might be a very
complicated formula, depending on the braiding. We give examples in some particular cases.
Lets denote $\sigma(x,y)=(\sigma^1\!(x,y), \sigma^2(x,y))$. 

 \begin{ex} In low degrees we have
 \begin{itemize}
  \item $d(e_x)=x-x'$
  \item $d(e_xe_y)=(e_zt-e_xy)-(x'e_y-z'e_t)$, where as usual $\sigma(x,y)=(z,t)$.
  \item $d(e_{x_1}e_{x_2}e_{x_3})=A_I-A_{II}$
  where
  
  $A_I=e_{\sigma^1\!(x_1,x_2)}e_{\sigma^1\!(\sigma^2(x_1,x_2),x_3)}\sigma^2(\sigma^2(x_1,x_2),x_3)-e_{x_1}e_{\sigma^1\!(x_2,x_3)}
  \sigma^2(x_2,x_3)+e_{x_1}e_{x_2}x_3$
  
  $A_{II}= x_1'e_{x_2}e_{x_3}-\sigma^1\!(x_1,x_2)'e_{\sigma^2(x_1,x_2)}e_{x_3}+
  \sigma^1\!(x_1,\sigma^1\!(x_2,x_3))'e_{\sigma^2(x_1,\sigma^1\!(x_2,x_3))}e_{\sigma^2(x_2,x_3)}$
  
  In particular, if $f:B\to k$ is an $A$-$A'$ linear map, then
  \[
f(d(e_{x_1}e_{x_2}e_{x_3}))=  
f(e_{\sigma^1\!(x_1,x_2)}e_{\sigma^1\!(\sigma^2(x_1,x_2),x_3)})
-f(e_{x_1}e_{\sigma^1\!(x_2,x_3)})+f(e_{x_1}e_{x_2})
\]\[-f(e_{x_2}e_{x_3})+f(e_{\sigma^2(x_1,x_2)}e_{x_3})-
f(e_{\sigma^2(x_1,\sigma^1\!(x_2,x_3))}e_{\sigma^2(x_2,x_3)})
\]
Erasing the $e$'s we notice the relation with the cohomological complex given in \cite{CES},
see Theorem \ref{teocomplejo} below.

   \end{itemize}
If  $X$ is a rack and $\sigma$ the braiding defined by  $\sigma(x,y)=(y,x\t y)=(x,x^y)$, then:
 \begin{itemize}
  \item $d(e_x)=x-x'$
  \item $d(e_xe_y)=
( e_{y}x^{y}
- e_{x}y)
-(x'e_{y}-y'e_{x^y})$
  \item $d(e_xe_ye_z)=
e_xe_yz
-e_xe_zy^z
+e_ye_zx^{yz}
-x'e_ye_z
+y'e_{x^y}e_z
-z'e_{x^z}e_{y^z}$.
\item In general, expressions I and II are
\[
 I=\sum_{i=1}^{n} (-1)^{i+1} e_{x_1}\dots e_{x_{i-1}}e_{x_{i+1}}\dots e_{x_n}x_{i}^{x_{i+1}\dots x_{n}}
\]
\[
 II=\sum_{i=1}^{n} (-1)^{i+1}x'_ie_{x_1^{x_i}}\dots e_{x_{i-1}^{x_i}}e_{x_{i+1}}\dots e_{x_n}
\]

then

\[
 \partial f(x_1,\dots,x_n)=f(d(e_{x_1}\dots e_{x_n}))=\]
\[
 \sum_{i=1}^{n} (-1)^{i+1} \left(f(x_1,\dots, x_{i-1},x_{i+1},\dots, x_n)x_{i}^{x_{i+1}\dots x_{n}}-x'_if({x_1}^{x_i},\dots
 , x_{i-1}^{x_i},x_{i+1},\dots, x_n)\right)
 \]
 
Let us consider $k\ot_{k[M']} B\ot_{k[M]} k$
then  $d$ represents the canonical differential of rack homology and 
$\partial f (e_{x_1}\dots e_{x_n})=f(d(e_{x_1}\dots e_{x_n}))$ gives the traditional rack cohomology structure. 

In particular, taking trivial coefficients:
 \[
 \partial f(x_1,\dots,x_n)=f(d(e_{x_1}\dots e_{x_n}))=\]
\[
 \sum_{i=1}^{n} (-1)^{i+1} \left(f(x_1,\dots, x_{i-1},x_{i+1},\dots, x_n)-f({x_1}^{x_i},\dots, x_{i-1}^{x_i},x_{i+1}\dots, x_n)\right)
\]

 \end{itemize}

\end{ex}

 \begin{teo}\label{teocomplejo}
Taking in $k$ the trivial $A'$-$A$-bimodule, the complexes associated to set theoretical Yang-Baxter solutions
defined in \cite{CES} can be recovered as
\[
 (C_\bullet(X,\sigma), \partial)\simeq (k\ot_{A'} B_\bullet \ot_{A} k, \partial=id_k\ot_{A'} d\ot_{A}id_k)
\]

\[
 (C^\bullet(X,\sigma), \partial^*)\simeq (\Hom_{A'-A}(B, k), \partial^*=d^*)
\]
 \end{teo}
In the  proof of the theorem  we will assume  first Proposition \ref{fnormal}
that says that one has a left $A'$-linear and right $A$-linear
isomorphism: 
\[B\cong A'\ot TE\ot A\]
where $A'=TX'/(x'y'=z't': \sigma(x,y)=(z,t))$ and $A=TX/(xy=zt: \sigma(x,y)=(z,t))$.
 We will prove 
 Proposition \ref{fnormal} later.

 \begin{proof}

In this setting every expression in $x,x',e_x$, using the relations defining $B$, can be written as
$x'_{i_1}\cdots x'_{i_n} e_{x_1}\cdots e_{x_k} x_{j_1}\cdots x_{j_l}$, tensorizing leaves the expression  
$$1\ot e_{x_1} \cdots e_{x_k}\ot 1$$
This shows that $T=k\ot_{k[M']} B\ot_{k[M]} k\simeq T\{e_x\}_{x\in X}$, where $\simeq$ means isomorphism of $k$-modules. 
This also induces isomorphisms of complexes

\[
 (C_\bullet(X, \sigma), \partial)\simeq (k\ot_{A'} B_\bullet \ot_{A} k, \partial= id_k\ot_{A'} d\ot_{A}id_k)
\]

\[
 (C^\bullet(X, \sigma), \partial^*)\simeq (\Hom_{A'-A}(B, k), d^*)
\]
 \end{proof}
 
Now we will prove Proposition \ref{fnormal}:
Call
$Y=\langle x,x',e_x\rangle_{x\in X}$ the free monoid in $X$ with unit 1, $k\langle Y\rangle$ the $k$ algebra associated to $Y$.
Lets define
$w_1=xy'$, $w_2=xe_y$ and $w_3=e_xy'$.
Let $S=\{r_1,r_2,r_3\}$ be the  reduction system  defined as follows: $r_i:k\langle Y\rangle\rightarrow k\langle Y\rangle$ the families of 
$k$-module endomorphisms
such that  $r_i$ fix all elements except 

$r_1(xy')=z't$,\ \ $r_2(xe_y)=e_zt$ and 
$r_3(e_xy')=z'e_t$. 

Note that $S$ has more than 3 elements, each $r_i$ is a family of reductions.
\

\begin{defi}
A reduction $r_i$ {\em acts trivially} on an element $a$ if $w_i$ does not apear in $a$, ie: $Aw_iB$ apears with coefficient 0.
\end{defi}

Following \cite{B}, $a\in k\langle Y\rangle$ is called  {\em irreducible} if $Aw_iB$ does not appear for $i\in\{1,2,3\}$. Call
$k_{irr}\langle Y\rangle$  the  $k$ submodule of irreducible elements of $k\langle Y\rangle$.
 A finite sequence of reductions is called {\em final} in $a$ if $r_{i_n}\circ \dots \circ r_{i_1}(a)\in k_{irr}(Y)$.
An element $a\in k\langle Y\rangle$ is called {\em reduction-finite} if for every sequence of reductions $r_{i_n}$ acts trivially on
  $r_{i_{n-1}}\circ \dots \circ r_{i_1}(a)$ for sufficiently large $n$.
   If a is reduction-finite, then any maximal sequence of reductions, such that  each $r_{i_j}$
   acts nontrivially on $r_{i_{(j-1)}}\dots r_{i_1}(a)$, will be finite, and hence a final 
sequence. It follows that the reduction-finite elements form 
a k-submodule of $k\langle Y\rangle$ 
 $a\in k\langle Y\rangle$ is called {\em reduction-unique} if is reduction finite and it's image under every finite
  sequence of reductions is the same.
  This comon value will be denoted $r_s(a)$.

\begin{defi}
 Given a monomial $a\in k\langle Y \rangle$ we define the disorder degree of 
 $a$, $\hbox{disdeg}(a)=\sum_{i=1}^{n_x}rp_i+\sum_{i=1}^{n_{x'}}lp_j$, where 
 $rp_i$ is the position of the $i$-th letter ``$x$'' counting from right to  left, and $lp_i$ is the position of the $i$-th letter ``$x'$'' 
 counting from left to right.
 
 If $a=\sum_{i=1}^{n} k_i a_i$ where $a_i$ are monomials in leters of $X,X', e_X$ and $k_i\in K-\{0\} $,
 \[\hbox{disdeg}(a):=\sum_{i=1}^{n}\hbox{disdeg}(a_i)\]
\end{defi}

\begin{ex}\begin{itemize}
           \item $\hbox{disdeg}(x_1e_{y_1}x_2z'_1x_3z'_2)=(2+4+6)+(4+6)=22$
           \item $\hbox{disdeg}(xe_yz')=3+3=6$ and $\hbox{disdeg}(x'e_yz)=1+1$
           \item $\hbox{disdeg}(\prod_{i=1}^{n}x'_i\prod_{i=1}^{m}e_{y_i}\prod^{k}_{i=1}z_i)=\frac{n(n+1)}{2}+\frac{k(k+1)}{2}$
           \end{itemize}

           \end{ex}

 The reduction $r_1$ lowers  disorder degree in two and reductions $r_2$ and $r_3$ lowers disorder degree in one. 
\begin{rem}\begin{itemize}
\item $k_{irr}(Y)=\{\sum A'e_B C: A' \ \hbox{word in}\  X', e_B \hbox{word in}\ e_x, C \ \hbox{word in}\ X\}$.
\item $k_{irr}\simeq TX'\ot TE\ot TX$
\end{itemize}
\end{rem}

Take for example $a=xe_yz'$, there are two possible sequences of final reductions: $r_3\circ r_1\circ r_2$ or $r_2\circ r_1\circ r_3$.
The result will be $a=A'e_B C$ and $a=D'e_E F$ respectively, where 

$A=\sigma^{(1)}\left(\sigma^{(1)}(x,y),\sigma^{(1)}(\sigma^{(2)}(x,y),z)\right)$

$B=\sigma^{(2)}\left(\sigma^{(1)}(x,y),\sigma^{(1)}(\sigma^{(2)}(x,y),z)\right)$

$C=\sigma^{(2)}\left(\sigma^{(2)}(x,y),z\right)$

$D=\sigma^{(1)}\left(x,\sigma^{(1)}(y,z)\right)$

$E=\sigma^{(1)}\left(\sigma^{(2)}(x,\sigma^{(1)}(y,z),\sigma^{(2)}(y,z))\right)$

$F=\sigma^{(2)}\left(\sigma^{(2)}(x,\sigma^{(1)}(y,z),\sigma^{(2)}(y,z))\right)$

We have $A=D$, $B=E$ and $C=F$ as $\sigma$ is a solution of YBeq, hence \\
$r_3\circ r_1\circ r_2(xe_yz')=r_2\circ r_1\circ r_3(xe_yz')$.

A monomial $a$ in $k\langle Y\rangle$ is said to have  an {\em overlap ambiguity} of $S$ if $a=ABCDE$
such that $w_i=BC$ and $w_j=CD$. We shall say the 
overlap ambiguity  is {\em resolvable} if there exist compositions of 
reductions, $r,r'$ such that $r(Ar_i(BC)DE)=r'(ABr_j(CD)E)$.
Notice that it is enough to take $r=r_s$ and $r'=r_s$.
\begin{rem}
In our case, there is only one type of overlap ambiguity and is the one we solved previously.
\end{rem}
\begin{proof}
 There is no rule with $x'$ on the left nor rule with $x$ on the right, so there will be no overlap ambiguity including the family $r_1$. 
There is only one type of ambiguity involving reductions $r_2$ and $r_3$.
 \end{proof}

Notice that  $r_s$  is a  proyector and  $I=\langle xy'-z't,xe_y-e_zt,e_xy'-z'e_t\rangle$ is trivially included in the kernel. 
We claim that it is actually equal:
\begin{proof}
 As $r_s$ is a proyector, an element  $a\in \ker$ must be $a=b-r_s(b)$ where $b\in k\langle Y\rangle$. It is enough to prove it for monomials $b$.
 \begin{itemize}
  \item if $a=0$ the result follows trivially.
  \item if not, then take a monomial  $b$ where at least one of the products $xy'$, $xe_y$ or $e_xy'$ appear. Lets suppose $b$ has a factor $xy'$
  (the rest of the cases are analogous). 
  
  $b=Axy'B$ where $A$ or $B$ may be empty words. $r_1(b)=Ar_1(xy')B=Az'tB$. Now we can rewrite:
  
  $b-r_s(b)=\underbrace{Axy'B-Az'tB}_{\in I}+Az'tB-r_s(b)$. 
  As $r_1$ lowers $\dd$ in two, we have $\dd(Az'tB-r_s(b))<\dd(b-r_s(b))$ then in a finite number of steps we get  $b=\sum^{N}_{k=1} i_k$
 where $i_k\in I$. It follows that $b\in I$.  
 \end{itemize}
 
\end{proof}
\begin{coro} $r_s$ induces a $k$-linear isomorphism: 
\[k\langle Y\rangle/\langle xy'-z't,xe_y-e_zt,e_xy'-z'e_t\rangle\rightarrow TX'\ot TE\ot TX \]
\end{coro}

Returning to our bialgebra, taking quotients we obtain the following proposition:

\begin{prop}\label{fnormal}
$B\simeq \left(TX'/(x'y'=z't')\right)\ot TE\ot \left(TX/(xy=zt)\right)$
\end{prop}

Notice that $\overline{x_1\dots x_n}=\overline{\prod \beta_m\circ \dots \circ\beta_1(x_1,\dots, x_n)}$ where $\beta_i=\sigma^{\pm 1}_{j_i}$,
analogously with $\overline{x'_1\dots x'_n}$.

This ends the proof of Theorem \ref{teocomplejo}.


\begin{ex}
 \item If the coeficients are trivial, $f\in C^1(X,k)$ and we identify
  $C^1(X,k)=k^X$, then
    \[
   (\partial f)(x,y)=f(d(e_xe_y))=-f(x)-f(y)+f(z)+f(t)
  \]
where as usual $\sigma(x,y)=(z,t)$ 
(If instead of considering $\Hom_{A'-A}$, we consider $\Hom_{A-A'}$ then 
 $(\partial f)(x,y)=f(d(e_xe_y))=f(x)+f(y)-f(z)-f(t)$ but with $\sigma(z,t)=(x,y)$).

 \item Again with trivial coefficients, and $\Phi\in C^2(X,k)\cong k^{X^2}$, then
    \[
   (\partial \Phi)(x,y,z)=\Phi (d(e_xe_ye_z))=\Phi\left(\overbrace{xe_ye_z}^{I}-\overbrace{x'e_ye_z}^{II}-\overbrace{e_xye_z}^{III}+
   \overbrace{e_xy'e_z}^{IV}+\overbrace{e_xe_yz}^{V}-\overbrace{e_xe_yz'}^{VI} \right) \]
   
   If considering  $\Hom_{A'-A}$ then,using the relations defining $B$, the terms $I,III,IV$ and $VI$ changes leaving
   \[
     \partial \Phi(x,y,z)=\Phi (\sigma^1\!(x,y),\sigma^1\!(\sigma^2(x,y),z))-\Phi(y,z)-\Phi(x,\sigma^1\!(y,z))+\]
     \[
    \Phi(\sigma^2(x,y),z)+\Phi(x,y)-\Phi(\sigma^2(x,\sigma^1\!(y,z)),\sigma^2(y,z))
   \]

 \item If $M$ is a  $k[T]$-module (notice that $T$ need not to be invertible as in \cite{tCES}) then
 $M$ can be viewed as an $A'-A$-bimodule via
 \[
  x' \cdot m=m
,\ \ m\cdot x=Tm
 \]
  The actions are compatible with the relations defining $B$:
  
  \[
   (m\cdot x)\cdot y =T^2 m \ ,\ \ (m\cdot z)\cdot t =T^2 m \]
   and 
   \[
  x'\cdot (y'\cdot m)=m \  ,  \ \  z'\cdot (t'\cdot m)= m   
  \]
Using these coefficients we get twisted cohomology as in \cite{tCES} but for general YB solutions.
 If one takes the special case of $(X,\sigma)$ being a rack, namely $\sigma(x,y)=(y,x\t y)$, then the general formula
 gives
 \[
 \partial f(x_1,\dots,x_n)=f(d(e_{x_1}\dots e_{x_n}))=\]
\[
 \sum_{i=1}^{n} (-1)^{i+1} \left(Tf(x_1,\dots, x_{i-1},x_{i+1},\dots, x_n)-f({x_1}^{x_i},\dots, x_{i-1}^{x_i},x_{i+1},\dots, x_n)\right)
 \]
 that agree with the differential of the twisted cohomology defined in \cite{tCES}.

\end{ex}

\begin{rem} 
If $c(x\ot y)=f(x,y)\sigma^1\!(x,y)\ot \sigma^2(x,y)$, then $c$ is  a solution of YBeq if and only if $f$ is a 2-cocycle.
\end{rem}

\[
 c_1\circ c_2\circ c_1(x\ot y\ot z)=
\]

\[
 =a\overbrace{\sigma^1\!\left(\sigma^1\!(x,y),\sigma^1\!(\sigma^2(x,y),z)\right)\ot \sigma^2\left(\sigma^1\!(x,y),\sigma^1\!(\sigma^2(x,y),z\right)\ot \sigma^2\left(\sigma^2(x,y),z)\right)}^{I}
\]
where
\[
 a=f(x,y)f\left(\sigma^2(x,y),z\right)f\left(\sigma^1\!(x,y),\sigma^1\!(\sigma^2(x,y),z)\right)
\]

\[
 c_2\circ c_1 \circ c_2 (x\ot y\ot z)=
\]

\[
 b\overbrace{\sigma^1\!(x,\sigma^1\!(y,z))\ot \sigma^1\!\left(\sigma^2(x,\sigma^1\!(y,z)),\sigma^2(y,z)\right)\ot \sigma^2\left(\sigma^2(x,\sigma^1\!(y,z),\sigma^2(y,z))\right)}^{II}
\]
where
\[b= f(y,z)f\left(x,\sigma^1\!(y,z)\right)f\left(\sigma^2(x,\sigma^1\!(y,z)),\sigma^2(y,z)\right)
\]
 Writing YBeq with this notation leaves:
\begin{equation}\label{trenza-trenza}
\sigma \ is\ a\ braid \Leftrightarrow    I=II
\end{equation}

Take $f$ a two-cocycle, then 

\[
 0=\partial f(x,y,z)=f(d(e_xe_ye_z))=f((x-x')e_ye_z-e_x(y-y')e_z+e_xe_y(z-z'))
 \]

 is equivalent to the following equality
 \[
  f(xe_ye_z)+f(e_xy'e_z)+f(e_xe_yz)=f(x'e_ye_z)+f(e_xye_z)+f(e_xe_yz')
 \]
using the relations defining $B$ we obtain

\[
 f\left(e_{\sigma^1\!(x,y)}e_{\sigma^1\!(\sigma^2(x,y),z)}\sigma^2(\sigma^2(x,y)z)\right)+f\left(\sigma^1\!(x,y)'e_{\sigma^2(x,y)}e_z\right)
 +f\left(e_xe_yz\right)
      \]
\[
 =f\left(x'e_ye_z\right)+f\left(e_xe_{\sigma^1\!(y,z)}\sigma^2(y,z)\right)+
 f\left(\sigma^1\!(x,\sigma^1\!(y,z))'e_{\sigma^2(x,\sigma^1\!(y,z))}e_{\sigma^2(y,z)}\right)
\]

If $G$ is an abelian multiplicative group and $f:X\times X\rightarrow (G, \cdotp)$ then the previous formula says 
\[
 f\left(e_{\sigma^1\!(x,y)}e_{\sigma^1\!(\sigma^2(x,y),z)}\sigma^2(\sigma^2(x,y)z)\right)f\left(\sigma^1\!(x,y)'e_{\sigma^2(x,y)}e_z\right)
 f\left(e_xe_yz\right)
\]
\[
 =f\left(x'e_ye_z\right)f\left(e_xe_{\sigma^1\!(y,z)}\sigma^2(y,z)\right)
 f\left(\sigma^1\!(x,\sigma^1\!(y,z))'e_{\sigma^2(x,\sigma^1\!(y,z))}e_{\sigma^2(y,z)}\right)
\]
which is exactly the condition $a=b$.

Notice that if the action is trivial, then the equation above simplifies giving
\begin{equation}\label{2-cocycle}
  f\!\left(e_{\sigma^1\!(x,y)}e_{\sigma^1\!   (\sigma^2(x,y),z)} \right)\! 
f \!\left(e_{\sigma^2(x,y)}e_z\right)\!
 f\!\left(e_xe_y\right)\!
 \newline
 =f\!\left(e_ye_z\right)\!f\!\left(e_xe_{\sigma^1\!(y,z)}\right)\!
 f\!\left(e_{\sigma^2(x,\sigma^1\!(y,z))}e_{\sigma^2(y,z)}\right)
\end{equation}
which is precisely the formula on \cite{CES} for Yang-Baxter 2-cocycles 
 (with $R_1$ and $R_2$ instead of $\sigma^1$ and $\sigma^2$).

\section{1st application: multiplicative structure
on cohomology}

\begin{prop} $\Delta$ induces an associative product in $\Hom_{A'-A}(B,k)$ (the graded Hom).
\end{prop}
 \begin{proof} It is clear that $\Delta$ induces an associative product on
 $\Hom_{k}(B,k)$ (the graded Hom), and
 $\Hom_{A'-A}(B,k)\subset\Hom_k(B,k)$ is a $k$-submodule. We will show that it is in fact a subalgebra.

 Consider the $A'$-$A$ diagonal structure on $B\ot B$
(i.e. $x_1'.(b\ot b').x_2=x_1'bx_2\ot x_1'b'x_2$)
and denote $B\ot^DB$ the $k$-module $B\ot B$ considered as $A'-A$-bimodule in this diagonal way.
We claim that 
$\Delta: B\rightarrow B\ot^D B$ is a morphism of $A'-A$-modules:
\[
 \Delta(x_1'yx_2)=x_1'yx_2\ot x_1'yx_2=x_1'(y\ot y)x_2
\]
same with $y'$, and with $e_x$:
\[
 \Delta(x_1'e_yx_2)=(x_1'\ot x_1')(y'\ot e_y+e_y\ot y)(x_2\ot x_2)=x'_1\Delta(e_y)x_2
\]

Dualizing $\Delta$ one gets:
\[
\Delta^*:\Hom_{A'-A}(B\ot^DB, k)\rightarrow \Hom_{A'-A}(B,k)
\]
consider the natural map
$$\iota:\Hom_k(B,k)\ot \Hom_k(B,k)\rightarrow \Hom_k(B\ot B, k)$$
$$\iota(f\ot g)(b_1\ot b_2)=f(b_1)g(b_2)$$
and denote $\iota|$ by
\[
 \iota|=\iota|_{\Hom_{A'-A}(B,k)\ot \Hom_{A'-A}(B,k)}
\]
Let us see that 
\[
 Im (\iota|)\subset \Hom_{A'-A}(B\ot B, k)\subset \Hom_k(B\ot B, k)
\]
If  $f,g: B\rightarrow k$ are two $A'-A$-module morphisms (recall $k$ has trivial actions, i.e. $x'\lambda=\lambda$ and $\lambda x=x$),
then
\[
 \iota(f\ot g)(x'(b_1\ot b_2))=f(x'b_1)g(x'b_2)=(x'f(b_1))(x'g(b_2))\]
 \[=f(b_1)g(b_2)=x'\iota (f\ot g)(b_1\ot b_2)
\]
\[
 \iota(f\ot g)((b_1\ot b_2)x)=f(b_1x)g(b_2x)=(f(b_1)x)(g(b_2)x)
 \]
 \[=(f(b_1)g(b_2))x=\iota(f\ot g)(b_1\ot b_2)x
\]
So, it is possible to compose $\iota|$ and $\Delta$, and obtain in this way an associative multiplication in $\Hom_{A'-A}(B,k)$.
\end{proof}

Now we will describe several natural quotients of $B$, each of them
give rise to a subcomplex of the cohomological complex of $X$
with trivial coefficients that are not only subcomplexes but also subalgebras;
in particular they are associative algebras.

\subsection{Square free case}
A solution $(X,\sigma)$ of YBeq satisfying $\sigma(x,x)=(x,x)\forall x\in X$ is called {\em square free}. 
For instance, if $X$ is a rack, then this condition is equivalent to $X$ being a quandle.

In the square free situation,
namely when $X$ is such that $\sigma (x,x)=(x, x)$ for all $x$, we add the condition $e_xe_x\sim 0$.

 If $(X,\sigma)$ is a square-free solution of the YBeq,  let us denote
 $sf$ the two sided ideal of $B$ generated by $\{e_xe_{x}\}_{x\in X}$.
\begin{prop}
\label{propsf}
$sf$ is a differential Hopf ideal. More precisely, 
\[
d(e_xe_{x})=0
\hbox{  and }
\Delta(e_xe_{x})=x'x'\ot e_xe_{x}+
e_xe_{x}\ot xx.\]
\end{prop}
In particular $B/sf$ is a differential graded bialgebra.
We may identify \\
$\Hom_{A'A}(B/sf,k)\subset 
\Hom_{A'A}(B,k)$ as the elements $f$ such that $f(\dots,x,x,\dots)=0$. If $X$ is a quandle, this
construction leads to the quandle-complex.
We have    $\Hom_{A'A}(B/sf,k)\subset 
\Hom_{A'A}(B,k)$ is not only a subcomplex, but also a subalgebra.

\subsection{Biquandles}

In \cite{KR}, a generalization of quandles is proposed (we recall it with different notation), 
a solution $(X,\sigma)$ is called
non-degenerated, or {\em birack} if in addition, 
\begin{enumerate}
 \item for any $x,z\in X$ there exists a unique $y$ such that $\sigma^1\!(x,y)=z$, (if this is the case,  $\sigma^1\!$ is called {\em left invertible}),
 \item for any $y,t\in X$ there exists a unique $x$ such that $\sigma^2(x,y)=t$, (if this is the case,  $\sigma^2$ is called {\em right invertible}),
\end{enumerate}
A birack is called {\em biquandle} if, given $x_0\in X$, there exists a unique $y_0\in X$ such that
$\sigma(x_0,y_0)=(x_0,y_0)$. In other words, if there exists a bijective map $s:X\to X$ such
that
\[
\{(x,y):\sigma(x,y)=(x,y)\}=
\{(x,s(x)): x\in X\}
\]

\begin{rem}
 Every quandle solution is a biquandle, moreover, given a rack $(X,\t)$, then
 $\sigma(x,y)=(y,x\t y)$ is a biquandle if and only if $(X,\t)$ is a quandle.  
\end{rem}

If $(X,\sigma)$ is a biquandle, for all $x\in X$ we add in $B$ the relation $e_{x}e_{s(x)}\sim 0$. 
Let us denote $bQ$ the two sided ideal of $B$ generated by $\{e_xe_{sx}\}_{x\in X}$.
\begin{prop}

\label{propbQ}
$bQ$ is a differential Hopf ideal. More precisely, 
$d(e_xe_{sx})=0$ and
$\Delta(e_xe_{sx})=x's(x)'\ot e_xe_{sx}+
e_xe_{sx}\ot xs(x)$.
\end{prop}
In particular $B/bQ$ is a differential graded bialgebra.
We may identify
\[
\Hom_{A'A}(B/bQ,k)\cong 
\{f\in \Hom_{A'A}(B,k) : f(\dots,x,s(x),\dots)=0\} 
\subset 
\Hom_{A'A}(B,k)\]
In \cite{CES}, the condition $f(\dots,x_0,s(x_0),\dots)=0$ is
 called the {\em type 1 condition}. A consequence of the above proposition is that
  $\Hom_{A'A}(B/bQ,k)\subset 
\Hom_{A'A}(B,k)$ is not only a subcomplex, but also a subalgebra.
Before proving this proposition we will review some other similar constructions.

\subsection{Identity case}
The two cases above may be generalized in the following way:

Consider $S\subseteq X\times X$ a subset of elements verifying
$\sigma(x,y)=(x,y)$ for all $(x,y)\in S$. 
Define $idS$ the two sided ideal of $B$ given by $idS=\langle e_xe_y/ (x,y)\in S\rangle$.

\begin{prop}
\label{propidS}
$idS$ is a differential Hopf ideal. More precisely, 
$d(e_xe_{y})=0$ for all $(x,y)\in S$ and
$\Delta(e_xe_{y})=x'y'\ot e_xe_{y}+
e_xe_{y}\ot xy$.
\end{prop}
In particular $B/idS$ is a differential graded bialgebra.\\
If one identifies $\Hom_{A'A}(B/sf,k)\subset 
\Hom_{A'A}(B,k)$ as the elements $f$ such that
\[ f(\dots,x,y,\dots)=0 \ \forall (x,y)\in S\]
We have that   $\Hom_{A'A}(B/idS,k)\subset 
\Hom_{A'A}(B,k)$ is not only a subcomplex, but also a subalgebra.
 
 \subsection{Flip case}
 Consider the condition $e_xe_y+e_ye_x\sim 0$ for all pairs such that $\sigma(x,y)=(y,x)$. For such a pair $(x,y)$ we have 
 the equations $xy=yx$, $xy'=y'x$, $x'y'=y'x'$ and $xe_y=e_yx$. Note that there is no equation for $e_xe_y$. 
 The two sided ideal $D=\langle e_xe_y+e_ye_x:\sigma(x,y)=(y,x)\rangle$ is a differential and Hopf ideal.

Moreover, the following generalization is still valid:

\subsection{Involutive case}

Assume  $\sigma(x,y)^2=(x, y)$. This case is called {\em involutive} in \cite{ETS}.
Define $Invo$ the two sided ideal of $B$ given by $Invo=\langle e_xe_y+e_ze_t : (x,y)\in X,\sigma(x,y)=(z,t)\rangle$.

\begin{prop}
\label{propInvo}
$Invo$ is a differential Hopf ideal. More precisely, 
$d(e_xe_{y}+e_ze_t)=0$ for all $(x,y)\in X$ (with $(z,t)=\sigma(x,y)$) and
if $\om=e_xe_y+e_ze_t$ then
$\Delta(\om)=x'y'\ot \om+\om \ot xy$.
\end{prop}
In particular $B/Invo$ is a differential graded bialgebra.
If one identifies \\ 
$\Hom_{A'A}(B/Invo,k)\subset 
\Hom_{A'A}(B,k)$ then
$\Hom_{A'A}(B/Invo,k)\subset \Hom_{A'A}(B,k)$ is not only a subcomplex, but a subalgebra.
 
\begin{conj}
$B/Invo$ is acyclic in positive degrees.
\end{conj}

\begin{ex} If $\sigma=flip$ and $X=\{x_1,\dots,x_n\}$ then $A=k[x_1,\dots,x_n]=SV$, the symmetric algebra on 
$V=\oplus _{x\in X}kx$. In this case
$(B/Invo,d)\cong (S(V)\ot\Lambda V\ot S(V),d)$ gives the Koszul resolution of $S(V)$ as
$S(V)$-bimodule.
\end{ex}

\begin{ex} If $\sigma=Id$,  $X=\{x_1,\dots,x_n\}$
and $V=\oplus _{x\in X}kx$, then $A=TV$ the tensor algebra.
If $\frac12\in k$, then
$(B/invo,d)\cong TV\ot (k\oplus V)\ot TV$ gives the Koszul resolution
of $TV$ as $TV$-bimodule. Notice that we don't really need $\frac12\in k$,
one could replace $invo=\langle e_xe_y+e_xe_y:(x,y)\in X\times X\rangle$ by
 $idXX=\langle e_xe_y:(x,y)\in X\times X\rangle$.
 \end{ex}

The conjecture above, besides these examples, is supported by
next result:
\begin{prop}\label{propinvo}
If $\Q\subseteq k$, then $B/Invo$ is acyclic in positive degrees.
\end{prop}
\begin{proof}
In $B/Invo$ it can be defined $h$ as the unique  (super)derivation such that:
\[
 h(e_x)=0; h(x)=e_x, h(x')=-e_x
\]
Let us see that $h$ is well defined:
\[
 h(xy-zt)=e_xy+xe_y-e_zt-ze_t=0
\]
\[
 h(xy'-z't)=e_xy'-xe_y+e_zt-z'e_t=0
\]
\[
 h(x'y'-z't')=-e_xy'-x'e_y+e_zt'+z'e_t=0
\]
\[
 h(xe_y-e_zt)=e_xe_y+e_ze_t=0
\]
Notice that in particular next equation shows
 that $h$ is not well-defined in $B$.
\[
 h(e_xy'-z'e_t)=e_xe_y+e_ze_t=0
\]
\[
 h(zt'-x'y)=e_zt'-ze_t+e_xy-x'e_y=0
\]
\[
h(ze_t-e_xy)= e_ze_t+e_xe_y=0
\]
\[
 h(e_zt'-x'e_y)=e_ze_t+e_xe_y=0
\]
\[
 h(e_xe_y+e_ze_t)=0
\]
Since (super) commutator of (super)derivations is again a derivation,
we have that $[h, d]=hd+dh$ is also a derivation. 
 Computations on generators: 
\[
h(e_x)=2e_x,\  h(x)=x-x', \ h(x')=x'-x
\]
or equivalently
\[
h(e_x)=2e_x,\  h(x+x')=0, \ h(x-x')=2(x-x')
\]
One can also easily see that $B/Invo$ is generated by
$e_x,x_{\pm}$, where $x_\pm=x\pm x'$, and that their
relations are homogeneous. We see that $hd+dh$ is nothing but the Euler
derivation with respect to the grading defined by
\[
\deg e_x=2,\
\deg x_+=0,\
\deg x_-=2,\]
We conclude automatically that the homology vanish for positive degrees of the $e_x$'s
(and similarly for the  $x_-$'s).
\end{proof}

 Next, we generalize Propositions 
 \ref{propsf}, \ref{propbQ}, 
 \ref{propidS} and \ref{propInvo}.
 
\subsection{Braids of order $N$}
 
 Let $(x_0,y_0)\in X\times X$ such that 
 $\sigma^N(x_0,y_0)=(x_0,y_0)$ for some $N\geq 1$.
 If $N=1$ we have the ''identity case'' and all  subcases, if $N=2$ we have the 
''involutive case''.
Denote
 \[
 (x_i,y_i):=\sigma^i(x_0,y_0) \ 1\leq i \leq N-1 
 \]
 Notice that the following relations hold in $B$:
\begin{itemize}
\item[$\star$] $x_{N-1}y_{N-1}\sim x_0y_0$, \ $x_{N-1}y'_{N-1}\sim x'_0y_0$, \  $x'_{N-1}y'_{N-1}=x'_0y'_0$
\item[$\star$] $ x_{N-1}e_{y_{N-1}}\sim e_{x_0}y_0$, \ $ e_{x_{N-1}}y'_{N-1}\sim x'_0e_{y_0}$
\end{itemize}
  and for $1\leq i \leq N-1$:
 \begin{itemize}
\item[$\star$] $x_{i-1}y_{i-1}\sim x_iy_i$, \  $x_{i-1}y'_{i-1}\sim x'_iy_i$, \  $x'_{i-1}y'_{i-1}=x'_iy'_i$
\item[$\star$] $ x_{i-1}e_{y_{i-1}}\sim e_{x_i}y_i$, \  $ e_{x_{i-1}}y'_{i-1}\sim x'_ie_{y_i}$
\end{itemize} 
Take $\om=\sum_{i=0}^{N-1}e_{x_i}e_{y_i}$, then we claim that 
\[
d\om=0\]
and 
\[\Delta\om
=x_0y_0\ot\om+\om\ot x'_0y'_0\]
For that, we compute
\[
 d(\om)=\sum_{i=0}^{N-1}(x_i-x'_i)e_{y_i}-e_{x_i}(y_i-y'_i)=
\]
\[
 \sum_{i=0}^{N-1}(x_ie_{y_i}-e_{x_i}y_i)-  \sum_{i=0}^{N-1} (x'_ie_{y_i}-e_{x_i}y'_i)=0
\]
For the comultiplication,
we recall that 
\[\Delta(ab)=\Delta(a)\Delta(b)\]
 where the product on the right hand side is defined using the Koszul sign rule:
 \[
 (a_1\ot a_2)(b_1\ot b_2)=(-1)^{|a_2||b_1|}a_1b_1\ot a_2b_2
 \]
So, in this case we have
\[
 \Delta(\om)=\sum_{i=0}^{N-1}\Delta (e_{x_i}e_{y_i})=
\]
\[
\sum_{i=0}^{N-1}(x'_iy'_i\ot e_{x_i}e_{y_i}-x'_ie_{y_i}\ot e_{x_i}y_i+ e_{x_i}y'_i\ot x_ie_{y_i}+e_{x_i}e_{y_i}\ot x_iy_i)
\]
the middle terms cancel telescopically, giving
\[
=\sum_{i=0}^{N-1}(x'_iy'_i\ot e_{x_i}e_{y_i}+e_{x_i}e_{y_i}\ot x_iy_i)
\]
and the relation $x_iy_i\sim x_{i+1}y_{i+1}$ gives
\[
=x'_0y'_0\ot( \sum_{i=0}^{N-1}e_{x_i}e_{y_i})+
(\sum_{i=0}^{n-1}e_{x_i}e_{y_i})\ot x_0y_0
\]
\[
=x'_0y'_0\ot \om+
\om\ot x_0y_0
\]
Then the two-sided ideal of $B$ generated by $\om$ is a Hopf ideal.
If instead of a single $\om$ we have several $\om_1,\dots \om_n$, we simply remark that
the sum of differential Hopf ideals is also a differential Hopf ideal.

\begin{rem} If X, is finite then for every $(x_0,y_0)$ there exists $N>0$ such that
$\sigma^N(x_0,y_0)=(x_0,y_0)$.
\end{rem}

\begin{rem}
Let us suppose $(x_0,y_0)\in X\times X$ is such that $\sigma^N(x_0,y_0)=(x_0,y_0)$ and  $u\in X$ an arbitrary element.
Consider the element 
\[
((\id \times \sigma)( \sigma \times \id) (u,x_0,y_0)=(\wt x_0,\wt y_0,u'')
\]
graphically
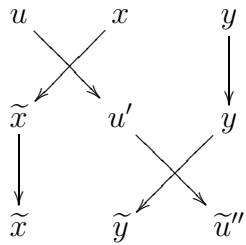
\[
 \xymatrix{
 u\ar[rd]&x\ar[ld]&y\ar[d]\\
 \wt x\ar[d]&u'\ar[rd]&y\ar[ld]\\
 \wt x&\wt y&\wt u'' 
 }
\]
then $\sigma^N(\wt x_0,\wt y_0)=(\wt x_0,\wt y_0)$.
\end{rem}
\begin{proof}\[
              (\sigma^N \times id)(\wt x_0,\wt y_0,u'')=(\sigma^N\times id)(id\times \sigma)(\sigma \times id)(u,x_0,y_0)=
             \]
\[
 (\sigma^{N-1}\times id)(\sigma\times id)(id\times \sigma)(\sigma \times id)(u,x_0,y_0)=
\]
using YBeq
\[
  (\sigma^{N-1}\times id)(id\times \sigma)(\sigma \times id)(id\times \sigma)(u,x_0,y_0)=
\]

repeating the procedure $N-1$ times leaves
\[
(id\times \sigma)(\sigma \times id)(id\times \sigma^N)(u,x_0,y_0)=(id\times \sigma)(\sigma \times id )(u,x_0,y_0)=(\wt x_0,\wt y_0,u'')
\]

\end{proof}

\section{$2^{nd}$ application: Comparison with Hochschild cohomology}

$B$ is a differential graded algebra, and on each degree $n$
it is isomorphic to $A\ot (TV)_n\ot A$, where $V=\oplus_{x\in X}ke_x$.
In particular  $B_n$
is free as $A^e$-module. We 
have {\em for free} the existence of a comparison map
\[
\xymatrix@-2ex{
\cdots\ar[r]&B_n\ar[r]\ar@{=}[d]&\cdots\ar[r]&B_2\ar[r]^d\ar@{=}[d]&B_1\ar[r]^d\ar@{=}[d]&\ar@{=}[d]B_0\\
\cdots\ar[r]&A'(TX)_n A\ar@{=}[d] ^{\cong}\ar[r]&\cdots\ar[r] &\oplus_{x,y\in X}A'e_xe_y A\ar@{=} ^{\cong}[d]\ar[r]^d&\oplus_{x\in X} A'e_x A\ar[r]^d\ar@{=}[d]&A' A\ar@{=}[d] ^{\cong}\\
\cdots\ar[r]&A\ot V^{\ot n}\ot A\ar[d]^{\wt\id}\ar[r]&\cdots\ar[r] &A\ot V^{\ot 2}\ot  A\ar[d]^{\wt\id}\ar[r]^{d_2}& A\ot V\ot A\ar[d]^{\wt\id}\ar[r]^{d_1}&A\ot A\ar[d]^{\id}\ar[r]^m& A\ar[d]^{\id}\ar[r]&0\\
\cdots\ar[r]&A\ot A^{\ot n}\ot A\ar[r]& \cdots\ar[r] &A\ot A^{\ot 2}\ot  A\ar[r]^{b'}& A\ot A\ot A\ar[r]^{b'}&A\ot A\ar[r]^m& A\ar[r]&0\\
}\]

\begin{coro}
For all $A$-bimodule $M$, there exists natural maps

\[
\wt\id_*: H^{YB}_\bullet(X,M)\to H_\bullet(A,M)
\]
\[
\wt\id^*: H^\bullet(A,M)\to H_{YB}^\bullet(X,M)
\]
that are the identity in degree zero and 1.
\end{coro}

Moreover, one can choose an explicit map with extra properties. For that we recall some definitions: there is a set theoretical section to the canonical projection  from the
Braid group to the symmetric group
\[
\xymatrix{
\BB_n\ar@{->>}[r]& \SS_n\ar@/_/@{..>}[l]
}
\]
\[
\xymatrix{
T_s:=\sigma_{i_1}\dots\sigma_{i_k}
&
 s=\tau_{i_1}\dots \tau_{i_k} \ar@{|->}[l]
}
\]
where 
\begin{itemize}
 \item $\tau\in S_n$ are transpositions of neighboring elements $i$ and $i+1$, 
so-called simple transpositions,
 \item $\sigma_i$ are the corresponding generators of $\BB_n$,
 \item $\tau_{i_{1}}\dots \tau_{i_{k}}$ is one of the shortest words representing $s$.
\end{itemize}
This inclusion factorizes trough 
\[
 \SS_n\hookrightarrow \BB_n^+\hookrightarrow  \BB_n
 \]
 It is a set inclusion not preserving the monoid structure.
 
\begin{defi}
The permutation sets
\[
 \Sh_{p_1,\dots,p_k}:=\left\{ s\in \SS_{p_1+\dots+p_k}/s(1)<\dots<s(p_1),\cdots, s(p+1)<\dots < s(p+p_k) \right\},
\]
where $p=p_1+\dots+p_{k-1}$,
are called {\em shuffle sets}.
\end{defi}

\begin{rem}
It is well known that  a braiding $\sigma$ gives an action of the positive braid monoid $B_n^+$ on $V^{\ot n}$, i.e. a monoid morphism
\[
 \rho: B_n^+\rightarrow End_\K(V^{\ot n})
\]
defined on generators $\sigma_i$ of $B_n^+$ by
\[
 \sigma_i\mapsto \id_V^{\ot(i-1)}\ot \sigma\ot \id_V^{\ot(n-i+1)}
\]

Then there exists a natural extension of a braiding in $V$ to a braiding in $T(V)$.

\[
 {\bf \sigma}(v\ot w)=(\sigma_k \dots \sigma_1)\circ\dots\circ(\sigma_{n+k-2}\dots\sigma_{n-1})\circ(\sigma_{n+k-1}\dots\sigma_n)(vw)\in V^{k}\ot V^{n}
\]
for $v\in V^{\ot n}$, $w\in V^{k}$ and $vw$ being the concatenation. 

Graphically

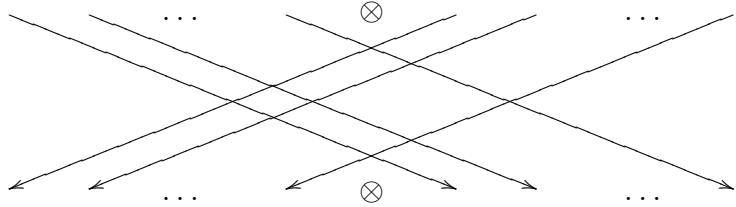
\[
\xymatrix{
\ar[rrrrrdd]&\ar[rrrrrdd]&\dots&\ar[rrrrrdd]&\ot&\ar[llllldd]&\ar[llllldd]&\dots&\ar[llllldd]\\
&&&&&&&&\\
&&\dots&&\ot&&&\dots&
}\]
\end{rem}

\begin{defi}
 The quantum shuffle multiplication on the tensor space $T(V)$ of a braided vector space $(V,\sigma)$ is the $k$-linear extension 
 of the map 
\[
 \shuffle_\sigma
=
\shuffle_\sigma^{p,q}:V^{\ot p}\otimes V^{\otimes q}\rightarrow V^{\ot(p+q)} 
\]
 \[
\ra v\ot \ra w \mapsto 
\ra v \shuffle_\sigma\ra w :=\sum_{s\in Sh_{p,q}}T_s^\sigma(\ra{vw})
 \]
Notation: $T_s^{\sigma}$ stands for the lift $T_s\in \BB_n^+$     
acting on $V^{\ot n}$ via the braiding $\sigma$. 
The algebra $Sh_\sigma(V):=(TV,\shuffle_\sigma)$ 
is called the quantum shuffle algebra on $(V,\sigma)$
\end{defi}
It is well-known that 
$\shuffle_\sigma$ is an associative product on $TV$(see for example
\cite{Le} for details) that makes it a Hopf algebra with deconcatenation
coproduct. 

\begin{defi}
Let $V$ be a braided vector space, then the quantum symmetrizer
map $\shuffle_{\sigma}:V^{\ot n}\to V^{\ot n}$ defined by
\[
QS_{\sigma}(v_1\otimes\cdots  \ot v_n)=
\sum_{\tau\in \SS_n}T^\sigma_\tau(v_1\ot\cdots\ot v_n)
\]
 where  $T_\tau^\sigma$ is the lift $T^\sigma_\tau\in \BB_{n}^+$ 
of $\tau$, acting on $V^{\otimes n}$ via the braiding $\sigma$.

In terms of shuffle products the quantum symmetrizer can be computed as
\[
\om \shuffle_{\sigma} \eta
:=\sum_{\tau\in\Sh_{p,q}}T^\sigma_\tau(\om\ot\eta)
\]

\end{defi}

The quantum symmetrizer map can also be defined as
\[
QS_\sigma(v_1\ot\cdots\ot v_n)=
v_1\shuffle_\sigma  \cdots \shuffle_\sigma v_n
\]
With this notation, next result reads as follows:
\begin{teo}\label{teoid}
 The  $A'$-$A$-linear  quantum symmetrizer map
\[
 \xymatrix{
 A'V^{\ot n}A\ar[r]^{\wt \id}& A\ot A^{\ot n}\ot A\\
 a_1'e_{x_1}\cdots e_{x_n}a_2\ar@{|->}[r]&
 a_1\ot(x_1\shuffle_{-\sigma}  \cdots \shuffle_{-\sigma} x_n)\ot a_2}
\]
is a chain map lifting the identify. Moreover, 
$\wt\id:B\to (A\ot TA\ot A,b')$ is a differential graded algebra
map, where in $TA$ the product is $\shuffle_{-\sigma}$, and
in $A\ot TA\ot A$ the multiplicative structure
 is not the usual tensor product algebra, but the braided one.
 In particular, this map factors through
 $A\ot \B\ot A$, where $\B$ is the Nichols algebra
 associated to the braiding $\sigma'(x\ot y)= - z\ot t$, where
 $x,y\in X$ and $\sigma(x,y)=(z,t)$. 
\end{teo}

\begin{rem}The Nichols algebra
$\B$ is the quotient of $TV$ by the ideal generated by  (skew)primitives that
are not in $V$, so the  result above explains the good behavior
of the ideals $invo$, $idS$,
or in general the ideal generated by
elements
of the form
 $\om=\sum_{i=0}^{N-1}e_{x_i}e_{y_i}$ where $\sigma(x_i,y_i)=(x_{i+1},y_{i+1})$ and $\sigma^N(x_0,y_0)=(x_0,y_0)$.
 It would be interesting to know the properties of $A\ot\B\ot A$
 as a differential object, since it appears to be a candidate of
 Koszul-type resolution for the semigroup algebra $A$
 (or similarly the group algebra $k[G_X]$).
 \end{rem}

The rest of the paper is devoted to the proof of \ref{teoid}. Most of the Lemmas
are "folklore" but we include them for completeness. The interested reader can
look at \cite{Lebed2} and references therein.

\begin{lem}\label{AC monoid}
 Let $\sigma$ be a braid in the braided (sub)category that contains two associative algebras  $A$ and $C$, meaning there 
 exists
 bijective functions 
 \[
\sigma_A:A\ot A\to A\ot A,\
\sigma_C:C\ot C\to C\ot C,\
\sigma_{C,A}:C\ot A\to A\ot C\]
such that  
\[
\sigma_*(1,-)=(-,1)\hbox{ and } \sigma_*(-,1)=(1,-) \ \hbox{ for } *\in \{A,C;C,A\}           
\]
\[
          \sigma_{C,A}\circ (1\otimes m_A)=(m_A\otimes 1)(1\otimes \sigma_{C,A})(\sigma_{C,A}\otimes 1)
\]
 and 
\[\sigma_{C,A}\circ ( m_C\otimes 1)=(1\otimes m_C)(\sigma_{C,A}\otimes 1)(1\otimes \sigma_{C,A})\]
Diagrammatically
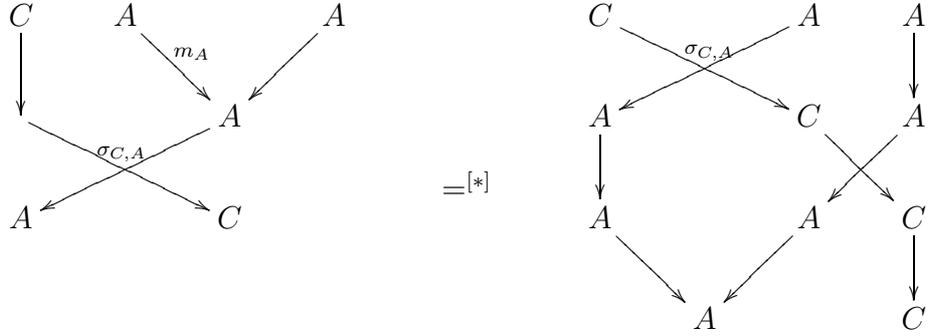
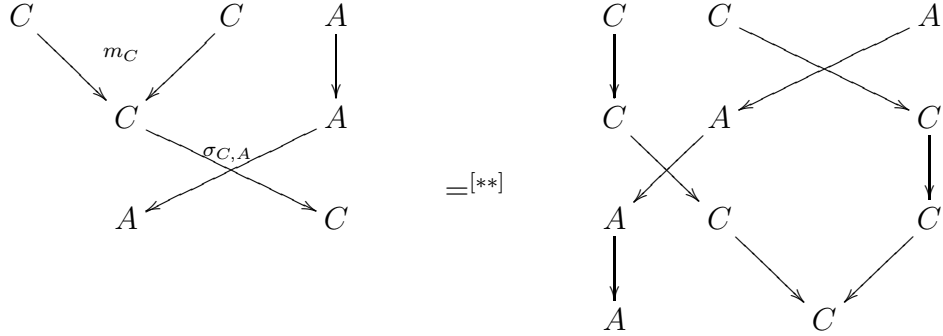
\[
\xymatrix{
C\ar[d]&A\ar[rd]^{\!\!\!m_A}&&A\ar[ld]\\
\ar[rrd]^{\!\!\!\!\!\!\! \sigma_{C,A}}&&A\ar[lld]&\\
A&&C&
}
\xymatrix{
\\
\\
&=^{[*]}&\\}
\xymatrix{
C\ar[rrd]^{\!\!\!\!\sigma_{C,A}}&&A\ar[lld]&A\ar[d]\\
A\ar[d]&&C\ar[rd]&A\ar[ld]\\
A\ar[rd]&&A\ar[ld]&C\ar[d]\\
&A&&C
}
\]
and
\[
\xymatrix{
C\ar[rd]^{\ \ m_C}&&C\ar[ld]&A\ar[d]\\
&C\ar[rrd]^{\!\!\!\!\!\!\!\sigma_{C,A}}&&A\ar[lld]\\
&A&&C
}
\xymatrix{
\\
\\
&=^{[**]}&\\}
\xymatrix{
C\ar[d]&C\ar[rrd]&&A\ar[lld]\\
C\ar[rd]&A\ar[ld]&&C\ar[d]\\
A\ar[d]&C\ar[rd]&&C\ar[ld]\\
A&&C&
}
\]
Assume that they
satisfy the braid equation with any combination of $\sigma_A,\sigma_C$ or $\sigma_{A,C}$.
Then, $A\ot_\sigma C=A\ot C$ with product defined by
 \[
(m_A\ot m_C)\circ(\id_A\ot \sigma_{C,A}\ot\id_C)\colon
(A\ot C)\ot (A\ot C)\to A\ot C
\]
is an associative algebra. In diagram:
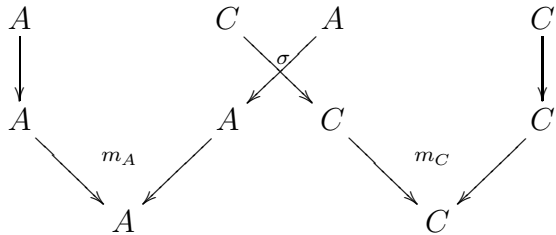
\[
\xymatrix{
A\ar[d]&&C\ar[rd]^{\!\!\!\sigma}&A\ar[ld]&&C\ar[d]\\
A\ar[rd]^{\ \ m_A}&&A\ar[ld]&C\ar[rd]^{\ \ m_C}&&C\ar[ld]\\
&A&&&C&
}\]
 
\end{lem}
\begin{proof}


Take $m\circ (1\otimes m)((a_1\otimes c_2)\otimes ((a_2\otimes c_2)\otimes (a_3\otimes c_3))$
 use $[*]$, associativity in $A$, associativity in $C$ then $[**]$ and the result follows.
\end{proof}

\begin{lem}
Let $M$ be the  monoid freely generated by $X$
module the relation $xy=zt$ where $\sigma(x,y)=(z,t)$, then,
$\sigma:X\times X\to X\times X$ naturally extends to a braiding in  $M$ and verifies


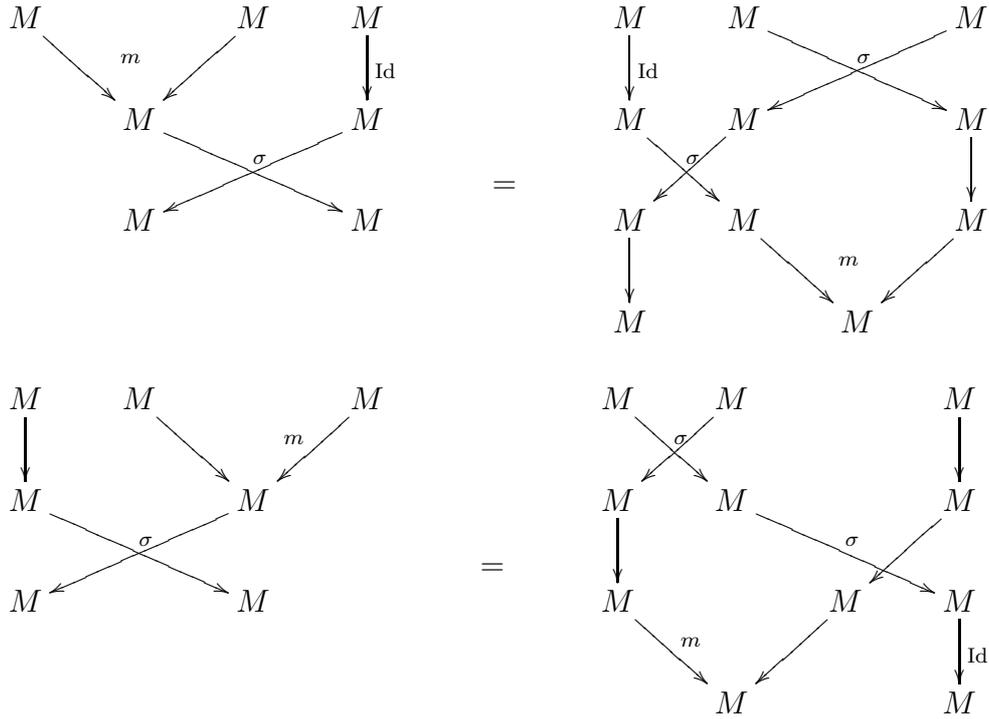
\[
\xymatrix{
M\ar[rd]^{\ \ \ m}&   &M\ar[ld]&M\ar[d]^\id\\
   &M\ar[rrd]^\sigma&&M\ar[lld]\\
&M&&M
} 
\xymatrix{
\\
\\
&=&\\
} 
\xymatrix{
M\ar[d]^{\id}   &M\ar[rrd]^\sigma&&M\ar[lld]\\
  M\ar[rd]^{\!\!\sigma}&M\ar[ld]&&M\ar[d]\\
M\ar[d]&M\ar[rd]^{\ \ \ m}&&M\ar[ld]\\
M&&M} 
\]

\[
\xymatrix{
M\ar[d]&M\ar[rd]   &&M\ar[ld]_{m}\\
  M \ar[rrd]^\sigma&&M\ar[lld]\\
  M&&M
} 
\xymatrix{
\\
\\
&=&\\
} 
\xymatrix{
M\ar[rd]^{\!\!\sigma}   &M\ar[ld]&&M\ar[d]\\
  M\ar[d]&M\ar[rrd]^\sigma&&M\ar[ld]\\
M\ar[rd]^m&&M\ar[ld]&M\ar[d]^\id\\
&M&&M} 
\]
\end{lem}

\begin{proof}
 It is enough to prove that the extension mentioned before is well defined in the quotient. Inductively, it will be enough to
 see that
 $\sigma(axyb,c)=\sigma(aztb,c)$ and  $\sigma(c,axyb)=\sigma(c,aztb)$ where 
$ \sigma(x,y)=(z,t)$, and this follows
 immediately from the braid equation:

 A diagram for the first equation is the following: 
  
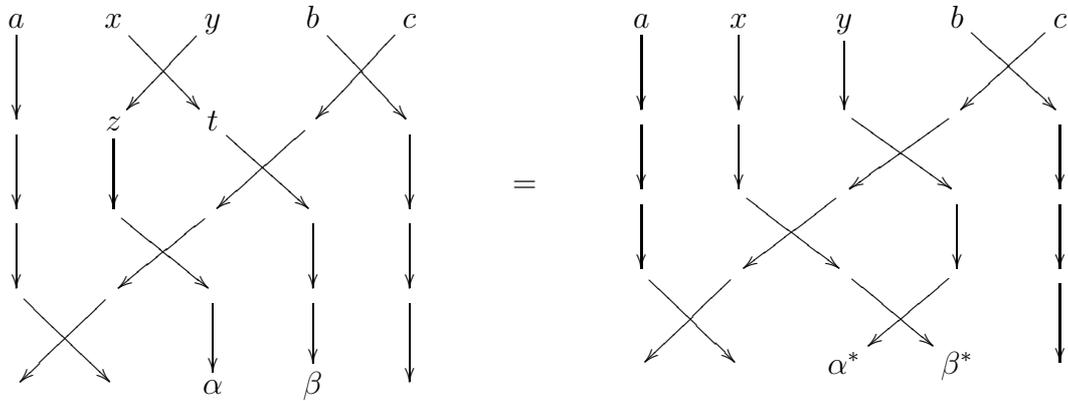
\[
\xymatrix{
a\ar[d]&x\ar[rd]&y\ar[ld]&b\ar[rd]&c\ar[ld]\\
\ar[d]&z\ar[d]&t\ar[rd]&\ar[ld]&\ar[d]\\
\ar[d]&\ar[rd]&\ar[ld]&\ar[d]&\ar[d]\\
\ar[rd]&\ar[ld]&\ar[d]&\ar[d]&\ar[d]\\
&&\alpha&\beta&} 
\xymatrix{
\\
\\
&=&\\
} 
\xymatrix{
a\ar[d]&x\ar[d]&y\ar[d]&b\ar[rd]&c\ar[ld]\\
\ar[d]&\ar[d]&\ar[rd]&\ar[ld]&\ar[d]\\
\ar[d]&\ar[rd]&\ar[ld]&\ar[d]&\ar[d]\\
\ar[rd]&\ar[ld]&\ar[rd]&\ar[ld]&\ar[d]\\
&&\alpha^*&\beta^*&
 } 
\]
 
 As $\alpha\beta=\alpha^*\beta^*$ the result follows.

\end{proof}

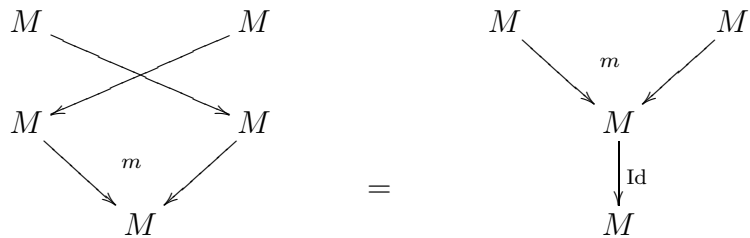
\begin{lem}

$m\circ\sigma=m$, diagrammatically:
\[
\xymatrix{
M\ar[rrd]&&M\ar[lld]\\
M \ar[rd]^{\ \ \ m}  &&M\ar[ld]\\
&  M } 
\xymatrix{
\\
\\
&=&\\
} 
\xymatrix{
M\ar[rd]^{\ \ \ m}&&M\ar[ld]\\
&  M\ar[d]^\id\\
&  M
 } 
\]

\end{lem}

\begin{proof} Using successively that $m\circ \sigma_i=m$, we have: 
\[m\circ \sigma(x_1\dots x_n, y_1\dots y_k)=m\left((\sigma_k\dots \sigma_1)\dots(\sigma_{n+k-1}\dots \sigma_n)_{(x_1\dots x_ny_1\dots y_k)}\right)\]
\[=m\left((\sigma_{k-1}\dots \sigma_1)\dots(\sigma_{n+k-1}\dots \sigma_n)_{(x_1\dots x_ny_1\dots y_k)}\right)=\dots\newline\]
\[
=m(x_1\dots x_n,y_1\dots y_k) 
\]
\end{proof}

\begin{coro}
If one considers  $A=k[M]$,
then the algebra $A$ verifies all diagrams in previous lemmas.
 \end{coro}

\begin{lem}
If $T=(TA, \shuffle_\sigma)$ there are bijective functions 
\[
\sigma_{T,A}:=\sigma|_{T\ot A}: T\ot A\rightarrow A\otimes T                                                                                    
\]
\[
\sigma_{A,T}:=\sigma|_{A\ot T}: A\ot T\rightarrow T\otimes A  
\]
that verifies the hypothesis of Lemma \ref{AC monoid}, and the same for 
 $(TA, \shuffle_{-\sigma})$.
\end{lem}
\begin{coro}
 $A\otimes (TA, \shuffle_{-\sigma})\otimes A$ is an algebra.
\end{coro}

\begin{proof}
 Use \ref{AC monoid} twice and the result follows.
\end{proof}

\begin{coro}\label{btp}
Taking $A=k[M]$, then the standard resolution of $A$ as $A$-bimodule has a natural algebra structure
 defining the braided  tensorial product as follows: 
\[
A\ot TA\ot A=
A\ot_\sigma(T^cA,\shuffle_{-\sigma})\ot_\sigma A\]
\end{coro}
Recall the differential of the standard resolution
is defined as
$b':A^{\ot n+1}\to A^{\ot n}$
\[ b'(a_0\ot\dots\ot a_n)= \sum_{i=0}^{n-1}(-1)^{i}a_0\ot \dots\ot a_ia_{i+1}\ot\dots\ot a_n\]
for all  $n\geq 2$.
If $A$ is a commutative algebra then the Hochschild resolution
is an algebra viewed as $\oplus_{n\geq 2}A^{\ot n}=A\ot TA\ot A$, with right and left 
$A$-bilinear extension  of the shuffle product  on $TA$, and $b'$ is a (super) derivation with
 respect to that product (see for instance Prop. 4.2.2 \cite{L}). 
In the braided-commutative case we have the analogous result:
\begin{lem}
$b'$ is a derivation with respect to the product mentioned in Corollary \ref{btp}.
\end{lem}
\begin{proof}
Recall the commutative proof 
as in Prop. 4.2.2 \cite{L}. 
Denote $*$ the product
 \[
(a_0\ot\dots\ot a_{p+1} )*(b_0\ot\dots\ot b_{q+1})=
a_0b_0\ot((a_1\dots\ot a_{p} )\shuffle (b_1\ot\dots\ot b_{q}))\ot a_{p+1}b_{q+1}
\]
Since $\oplus_{n\geq 2}A^{\ot n}=A\ot TA\ot A$
is generated by $A\ot A$ and $1\ot TA\ot 1$, we check on generators.
For $a\ot b\in A\ot A$, $b'(a\ot b)=0$, in particular, it satisfies Leibnitz
rule for elements in $A\ot A$.
 Also, $b'$ is $A$-linear on the left, and right-linear on the right, so
\[
 b'\big((a_0\ot a_{n+1})*(1\ot a_1\ot \cdots \ot a_n\ot 1)\big)=
b'(a_0\ot a_1\ot\cdots\ot a_n\ot a_{n+1})
\]
\[=
a_0b'(1\ot a_1\ot\cdots\ot a_n\ot 1)a_{n+1}
=(a_0\ot a_{n+1})*b'(1\ot a_1\ot\cdots\ot a_n\ot 1)
\]
\[
=0+(a_0\ot a_{n+1})*b'(1\ot a_1\ot\cdots\ot a_n\ot 1)
 \]\[
=b'(a_0\ot a_{n+1})*(1\ot a_1\ot\cdots\ot a_n\ot 1)
+(a_0\ot a_{n+1})*b'(1\ot a_1\ot\cdots\ot a_n\ot 1)
\]
Now consider $(1\ot a_1\ot \dots\ot a_{p}\ot 1 )*(1\ot b_1\ot\dots\ot b_{q}\ot 1)$,
it is a sum
of terms where two consecutive tensor terms can be of the form
$(a_i,a_{i+1})$, or $(b_j,b_{j+1})$, or $(a_i,b_j)$ or $(b_j,a_i)$.
When one computes $b'$, multiplication of two consecutive tensor factors will give,
respectively,  terms  of the form
\[ \cdots\ot a_ia_{i+1}\ot \cdots, \  \cdots\ot b_jb_{j+1}\ot\cdots,\ \cdots\ot a_ib_j\ot\cdots,
\cdots  \ot b_ja_i\ot\cdots
\]
The first type of terms will recover $b'((1\ot a_1\ot\cdots\ot a_n\ot 1))*
(1\ot b_1\ot\cdots\ot b_q\ot 1)$ and the second type of terms will recover
$\pm (1\ot a_1\ot\cdots\ot a_n\ot 1)*b'((1\ot b_1\ot\cdots\ot b_q\ot 1))$.
On the other hand, the difference between the third and forth type of terms is just
a single trasposition so they have different signs, while $a_ib_j=b_ja_i$ because the 
algebra is commutative, if one take the {\em signed} shuffle then  they cancel each other.

In the {\em braided} shuffle product, the summands are indexed by the same set
of shuffles, so we have the same type of terms, that is, when computing $b'$ of 
a (signed) shuffle product, one may do the product of two elements in coming form the 
first factor, two elements of the second factor. 
or a mixed term. For the mixed terms, they will have the form
\[
\cdots\ot A_iB_j \ot\cdots \hbox{, or }
\cdots\ot \sigma^1(A_i,B_j)\sigma^2(A_i,B_j)\ot\cdots
    \]
As in the algebra $A$ we have  $A_iB_j=\sigma^1(A_i,B_j)\sigma^2(A_i,B_j)$ then this 
terms will cancel leaving only the terms corresponding to 
 $b'(1\ot a_1\ot\cdots\ot a_p \ot 1)\shuffle_{-\sigma} (1\ot b_1\ot\cdots\ot b_q\ot )$
and $\pm(1\ot a_1\ot\cdots\ot a_p\ot 1 )\shuffle_{-\sigma} b'(1\ot b_1\ot\cdots\ot b_q\ot 1)$
respectively.


\end{proof}

\begin{coro} There exists a  comparison morphism  
$f:(B,d)\to (A\ot TA\ot A,b')$ which is a differential graded algebra morphism, $f(d)=b'(f)$,
 simply defining it on  $e_x$ ($x\in X$)
 and verifying $f(x'-x)=b'(f(e_x))$.
\end{coro}
\begin{proof}
Define $f$ on $e_x$, extend $k$-linearly to $V$, multiplicatively to $TV$, and $A'$-$A$ linearly to
$A'\ot TV\ot A=B$. In order to see that $f$ commutes with the differential, by $A'$-$A$-linearity
it suffices to check on $TV$, but since $f$ is multiplicative on $TV$ it is enough to check on $V$, and by $k$-linearity we check on basis, that is, we only need $f(de_x)=b'f(e_x)$.
\end{proof}

\begin{coro}
$f|_{TX}$ is the quantum symmetrizer map, and therefore 
$\Ker(f)\cap TX\subset B$ defines the  Nichol's ideal 
associated to $-\sigma$.
\end{coro}
\begin{proof}
\[
 f(e_{x_1}\cdots e_{x_n})=f(e_{x_1})*\cdots *f(e_{x_n})=(1\ot x_1\ot 1)*\cdots *(1\ot x_n \ot 1)=1\ot(x_1\shuffle \cdots \shuffle x_n)\ot 1
\]

\end{proof}

The previous corollary explains why $\Ker(\id-\sigma)\subset B_2$
gives a Hopf ideal and also ends the proof of Theorem \ref{teoid}.

\begin{question}
$Im(f)=A\ot \mathfrak{B}\ot A$ is a resolution of $A$ as a $A$-bimodule? namely,
is $(A\ot\B\ot A,d)$ acyclic?
\end{question}
This is the case for involutive solutions in characteristic zero, but
 also for $\sigma=$flip in any characteristic, and $\sigma=\id$ (notice this $\id$-case 
gives the Koszul resolution for the tensor algebra). If the answer to that question is yes,
and $\B$ is  finite dimensional then $A$ have necessarily finite global dimension. 
Another interesting question is how to relate 
generators for the relations defining  $\B$ and cohomology classes for $X$.

\end{document}